\documentclass[10pt]{article}
\usepackage{graphicx}
\usepackage{latexsym}
\usepackage{xcolor}
\usepackage{amsmath,amsthm,amssymb}
\usepackage[hypcap]{caption}
\usepackage[all]{xy}
\SelectTips{cm}{12}
\objectmargin={1pt}
\usepackage[textwidth=14.5cm]{geometry}

\usepackage{tabularx}
\makeatletter
 
 \@addtoreset{equation}{section}
\makeatother
\makeatletter
\newcommand{\subjclass}[2][1991]{%
  \let\@oldtitle\@title%
  \gdef\@title{\@oldtitle\footnotetext{#1 \emph{Mathematics subject classification.} #2}}%
}
\newcommand{\keywords}[1]{%
  \let\@@oldtitle\@title%
  \gdef\@title{\@@oldtitle\footnotetext{\emph{Key words and phrases.} #1.}}%
}
\makeatother

\theoremstyle{plain}
\newtheorem{df}{\sc \bf Definition}[section]
\newtheorem{thm}[df]{\sc \bf Theorem}

\newtheorem{prop}[df]{\sc \bf Proposition}

\theoremstyle{remark}
\newtheorem{rem}[df]{\sc Remark}

\newcommand{\K}{K}

\newcommand{\s}{\mathfrak{s}}

\newcommand{\R}{\boldsymbol{R}}
\newcommand{\bH}{\boldsymbol{H}}

\newcommand{\tr}{\mathrm{tr}}

\def\l{{\mathfrak{l}}}

\newcommand{\n}{\mathfrak{n}}
\newcommand{\q}{\mathfrak{q}}
\newcommand{\p}{\mathfrak{p}}
\newcommand{\m}{\mathfrak{m}}
\newcommand{\g}{\mathfrak{g}}
\newcommand{\kk}{\mathfrak{k}}
\newcommand{\da}{\mathfrak{a}}

\newcommand{\sll}{\mathfrak{sl}}
\newcommand{\gl}{\mathfrak{gl}}

\usepackage{enumerate}
\usepackage{multirow,bigdelim}
\usepackage{delarray}

\begin{document}
\title{Projectively flat and affinely flat parabolic subgroups of special linear groups}


\author{Hironao Kato 
\thanks{During the research the author was supported by JSPS and JSPS Strategic Young Researcher Overseas Visits Program for  Accelerating Brain Circulation.}}

\maketitle
\begin{abstract}
A special linear Lie group over the real number field and the quarternion field admits a projectivley flat affine connection. 
We show that parabolic subgroups are autoparallel submanifolds and give a criterion the induced connection is projectively equivalent to a flat affine connection.  
\end{abstract}
\section{Introduction} 

An affine connection on a manifold is projectively flat if the connection is locally projectively equivalent to a flat affine connection.  Thus a flat affine connection is also projectively flat.  On any Lie group $L$ we can consider a left invariant affine connection,  however $L$ does not necessarily admit projectively flat affine connections. In fact any Lie group of dimension 
$\leq$ 5 admits a projectively flat affine connection (see \cite{hkato1}),  however 
6 dimensional real semisimple Lie group such as $O(3, \R) \times O(3, \R)$ and $SL(2, \R) \times SL(2, \R)$ does not admit them (cf. \cite{agaoka-kato}).  
We consider the existence problem of left invariant projectively flat affine connections on Lie groups, which is widely open.
This problem is closely related to classification of prehomogeneous vector spaces (abbrev. PVs) and left symmetric algebras (cf. \cite{hkato2}). 

      
In particular from the viewpoint of submanifolds 
we use an projectively flat affine connection on Lie groups with Lie algebra $\sll(n, \R)$ or $\sll(n, \bH)$, which was constructed by Agaoka \cite{agaoka}.  These  are the only simple Lie algebras which admit left invariant projectively flat affine connections (see Urakawa \cite{urakawa}, Elduque \cite{elduque}). 
The special linear Lie algebras equipped with the projectively flat affine connection correspond to associative algebras with identity, which was proved by Nomizu and Pinkall \cite{nomizu-pinkall}. 
Associative algerbra with identities are special classes of infinitesimal PVs.

We remark that semisimple Lie groups do not admit flat affine connections. 
However on the borel subalgebra of semisimple Lie algebra a left invariant flat affine connection was intrinsically 
constructed by Takemoto and Yamaguchi \cite{takemoto-yamaguchi}.    
In this article with respect to the connection on special linear groups 
we investigate their parabolic subgroups and 
solvable Lie subgroups associated to the Langlands decomposition, and show that they are autoparallel submanifolds.  
The construction of those subgroups is adapted from Tamaru \cite{tamaru}.  
In the main theorem of the article we give a criteria that the induced affine connection is projectively equivalent to a flat affine connection.




\section{Preliminaries} 

Let $L$ be a Lie group of dimension $n$ and $\l$ its Lie algebra. 
Let $\nabla$ and $\nabla'$ be torsion-free affine connections on $L$. 
Connections $\nabla$ and $\nabla'$  are projectively equivalent if there exists a 1-form $\lambda$ on $L$ such that $\nabla_X Y - \nabla_X Y = \lambda(X)Y + \lambda(Y)X$ for vector fields $X$ and $Y$. 
A projective equivalence class $[\nabla]$ of torsion-free affine connection $\nabla$ is called a projective structure.     
The connection $\nabla$ is affinely flat if its curvature tensor vanishes, i.e. 
$R(X, Y)Z = \nabla_X\nabla_YZ - \nabla_Y\nabla_XZ - \nabla_{[X, Y]}Z = 0$.  
The connection $\nabla$ is projectivley flat if Weyl's projective curvature tensor vanishes for $n \geq 3$, i.e. 
$W(X, Y)Z = R(X, Y)Z + [P(X,Y)- P(Y, X)]Z - [P(Y, Z)X - P(X, Z)Y]=0$ (cf. \cite{nomizu-sasaki}). 
For $n = 2$, $\nabla$ is projectivley flat if $\nabla_X P(Y, Z) = \nabla_Y P(X, Z)$. 
Here $P$ is the $(1, 1)$-tensor defined by  $P(X, Y) = \frac{1}{n^2-1}[n Ric(X, Y) + Ric(Y, X)]$. 
If $\nabla$ is left invariant projectively flat, then $[\nabla]$ is called a left invariant flat projective structure. 
 
In \cite{agaoka} Agaoka defined a Lie algebra homomorphism $\l \to \sll(n+1, \R)$ called a (P) - homomorphism, and  
established the following bijection: 
$\{$Left invariant projectively flat affine connection $\nabla$ on $L$$\}$ $\to$ 
$\{$(P)-homomorphism $f: \l \to \sll(n+1, \R)$$\}$ via Cartan connections.   
Denote by  
$\{e_{1}, \ldots, e_{n+1}\}$ the standard basis of $\R^{n+1}$ and by $\{X_1, \ldots, X_n\}$ a basis of $\l$. 
Then a Lie algebra homomorphism $f: \l \to \sll(n+1, \R)$ is called a (P)-homomorphism if $f(X_i)e_{n+1} = e_i + \alpha e_{n+1}$ for some $\alpha \in \R$. 
We can directly prove the correspondence as follows.  As a result (P)-homomorphism $f$ corresponding to $\nabla$ is given by  
\begin{equation} \label{form of rep1} 
f(X) = 
\begin{pmatrix} 
\nabla_X - \frac{1}{n+1}tr\nabla_X I_n & X \\
-P(X, \cdot) & - \frac{1}{n+1}tr\nabla_X
\end{pmatrix}. 
\end{equation}
Denote by $f_1$ the $(2, 1)$-part of $f$ with respect to the $2 \times 2$ block decomposition. 
Put $\g = \sll(n+1, \R)$ and we call $f_1$ the $\g_1$ component of $f$. 
Then we have 
$f_1(X) = -P(X, \cdot)$.     
Indeed we can see that 
torsion free affine connection $\nabla$ is projectively flat iff a linear map $f: \l \to \gl(n+1, \R)$ defined by 
(\ref{form of rep1}) is a Lie algebra homomorphism: 
\begin{eqnarray*}
&& [f(X), f(Y)]Z - f([X, Y])Z \\ 
&=&  
\{[
\begin{pmatrix} 
\nabla_X  & X \\
-P(X, \cdot) & 0
\end{pmatrix},  
\begin{pmatrix} 
\nabla_Y  & Y \\
-P(Y, \cdot) & 0
\end{pmatrix}
]  
- 
\begin{pmatrix} 
\nabla_{[X, Y]} - \frac{1}{n+1}\tr\nabla_{[X, Y]} I_n & [X, Y] \\ 
-P([X, Y], \cdot) & \hspace{-5mm} - \frac{1}{n+1}\tr\nabla_{[X, Y]} 
\end{pmatrix}\} 
\begin{pmatrix}
Z \\ 
0 
\end{pmatrix} \quad \quad 
\\ 
&=&  
\begin{pmatrix} 
\nabla_X \nabla_Y Z - P(Y, Z)X \\
-P(X, \nabla_Y Z)
\end{pmatrix}
- 
\begin{pmatrix} 
\nabla_Y \nabla_X Z - P(X, Z)X \\
-P(Y, \nabla_X Z)
\end{pmatrix}
-
\begin{pmatrix} 
\nabla_{[X, Y]} Z - \frac{\tr\nabla_{[X, Y]}}{n+1}Z \\
-P([X,Y], Z)
\end{pmatrix} \\
&=&  
\begin{pmatrix} 
W(X, Y)Z \\
-P(X, \nabla_Y Z) + P(Y, \nabla_X Z) + P([X,Y], Z)  
\end{pmatrix}. 
\end{eqnarray*}
Here we used the equality $\tr \nabla_{[X, Y]} = - \tr R(X, Y) = (n+1)[P(X, Y)- P(Y, X)]$. 
The second row vanishes iff  
we have the Codazzi's equation   
$\nabla_X P(Y, Z) = \nabla_Y P(X, Z)$. Recall that we have the Codazzi's equation if the Weyl's projective curvature vanishes.  
We also have the following equality: 
\[[f(X), f(Y)]v - f([X, Y])v  =  
\begin{pmatrix} 
\nabla_X Y - \nabla_Y X - [X, Y] \\
-P(X, Y) + P(Y, X) + \frac{\tr \nabla_{[X, Y]}}{n+1}
\end{pmatrix}. 
\] 

Thus $\nabla$ is projectively flat iff $f$ is a Lie algebra homomorphism.  
This proof is a generalization of the proof in \cite{agaoka-kato}, which is dealing with the case of Ricci symmetric. 
When $\nabla$ is Ricci symmetric,  the corresponding homomorphism $f$ is of the form 
\begin{equation} 
f(X) = 
\begin{pmatrix} 
\nabla_X - \frac{1}{n+1}tr\nabla_X & X \\
-\frac{1}{n-1}Ric(X, )  & - \frac{1}{n+1}tr\nabla_X
\end{pmatrix}.
\end{equation} 
We denote by $\gamma(X, Y)$ the normalized Ricci tensor $\frac{1}{n-1}Ric(X, Y)$.   
In this case we can modify a linear map $f$ by   
\begin{equation} \label{form of rep}
f'(X) = 
\begin{pmatrix} 
\nabla_X  & X \\
-\gamma(X, )  & 0 
\end{pmatrix} 
\end{equation} 
so that $f': \l \to \gl(n+1, \R)$ gives again a Lie algebra representation.  Conversely if $f: \l \to \gl(n+1, \R)$ is a representation of the form 
$f(X) = 
\begin{pmatrix}
\nabla_X & X \\ 
f_1(X, \cdot) & 0 
\end{pmatrix}$, then we have $f_1(X, Y) = f_1(Y, X) = - \gamma(X, Y)$. Thus Ricci tensor is symmetric. 
Moreover if $\nabla$ is affinely flat, then Ricci tensor vanishes and $f_1 = 0$. 
Conversely if $f_1 = 0$, then we can directly prove that $\nabla$ is affinely flat (cf. Theorem 3.7 of \cite{agaoka}).  

Let $\{e_1, \ldots, e_{n+1}\}$ be the standard basis of $\R^{n+1}$. 
From the expression (\ref{form of rep1}) of (P)-homomorphism $f: \l \to \sll(n+1, \R)$ 
we can easily see that $f$ satisfies the condition 
$f \otimes \Lambda_1 (\l \oplus \R)e_{n+1} = \R^{n+1}$. Thus $f$ gives a representation called 
an infinitesimal prehomogeneous vector space.  Conversely a given PV $f: \l \oplus \R \to \gl(V)$ such as $\dim \l +1 = \dim V$ 
gives a (P)-homomorphism as follows:

Let $v$ be a generic point of a PV $(\l \oplus \R, f \otimes \Lambda_1, V)$. Then we define a matrix $P$ to be 
$(f(X_1)v, \cdots, f(X_n)v, v)$. Consider the projection $\gl(n+1, \R) \to \sll(n+1, \R)$ and denote its image of $f$ by $\bar{f}$.   
Then $P^{-1} \bar{f} P$ gives a (P)-homomorphism $\l \to \sll(n+1, \R)$.  
The $\g_0$ part of the (P)-homomorphism $P^{-1} \bar{f} P$ gives a left invariant 
projectively flat affine connection. 

There is a unique Lie algebra representation of $\l$ called a (N)-homomorphism which is 
projectively equivalent to $P^{-1} \bar{f} P$ (cf. \cite{agaoka}). This (N)-homomorphism is explicitly given by 
$\bar{P}^{-1} \bar{f} \bar{P}$, where $\bar{P}$ is the matrix $(\bar{f}(X_1)v, \cdots, \bar{f}(X_n)v, v)$. 
We denote $\bar{P}^{-1} \bar{f} \bar{P}$ by $f_v$. 
Now we introduce an equivalence relation. 
We denote $(g, w) \sim (f, v)$ iff there exists $Q  \in GL(n+1, \R)$  such that $\bar{g} = Q \bar{f} Q^{-1}$ and 
$w = Q v$.  In \cite{hkato1} the author 
proved the correspondence between invariant flat complex projective structures on complex Lie groups and the 
equivalence classes of pairs of infinitesimal prehomogeneous vector spaces and their generic points. 
By considering the correspondence over the real number fields we obtain  
the following one-to-one correspondence:  
\begin{eqnarray*} 
&& \{\mbox{Left invariant  flat projective structure on} \  L\} \\ 
&\to & \{(f, v) \mid (\l \oplus \R, f \otimes \Lambda_1, V \otimes \R) \ \mbox{is a PV s.t.} 
 \dim l +1 = \dim V \}/_\sim. 
\end{eqnarray*}   
Note that we can also directly prove the above one-to-one correspondence. 
Furthermore if $(f, v) \sim (g, w)$, then we have $f_v = g_w$. Hence from this correspondence we can recover the one-to-one correspondence in \cite{agaoka} between left invariant flat projective structures on Lie groups and (N)-homomorphisms. 



Now we state an easy but important fact.  
\begin{prop}\label{autoparallel}
Let $(L, \nabla)$ be a left invariant projectively flat Lie group and $S$ be a Lie subgroup of $L$. 
If $S$ is an autoparallel submanifold, then the induced left invariant affine connection on $S$ is projectively flat.  
\end{prop} 
\begin{proof} 
Denote by $\nabla^S$ and $R^S$ the induced connection on $S$ and its curvature tensor respectively. 
Then we have  $R^S(X, Y)Z = - [P(X,Y)- P(Y, X)]Z + [P(Y, Z)X - P(X, Z)Y]$  for left invariant vector fields $X, Y, Z$ on $S$ (cf. the appendix of \cite{nomizu-sasaki}).  
Denoting by $m$ the dimension of $S$, we have 
$Ric^S(X, Y) = m P(X, Y) - P(Y, X) = \frac{1}{n^2-1}[(mn-1)Ric(X,Y) + (m-n)Ric(Y, X)]$. 
It follows that 
$P^S(X, Y) = \frac{1}{m^2-1}[m Ric^S(X, Y) + Ric^S(Y, X)]
= P(X, Y)$. 

Then the Weyl's projective curvature $W^S$ of $S$ coincides with the restriction of $W$, i.e.  
$W^S(X, Y)Z = W(X, Y)Z$. 
Hence $(S, \nabla^S)$ is a projectively flat Lie subgroup. 
\end{proof}
\begin{rem}
Denote by $\gamma^S$ the normalized Ricci tensor of $(S, \nabla^S)$.  
In Proposition \ref{autoparallel} if $\nabla$ is Ricci symmetric, then $\gamma^S = \gamma|_S$. 
\end{rem}

\section{projectively flat Lie subgroups} 
In this section let us begin by recalling the parabolic subalgebras and the Iwasawa decomposition of semisimple Lie algebras,  following \cite{tamaru}, which the reader can consult for detail.  
Let $\g$ be a Lie algebra and $\sigma$ a Cartan involution. Denote by $\kk$ and $\p$ an eigenspace with eigenvalue 1 and 
 $-1$ respectively. Then we have the Cartan decomposition $\g = \kk \oplus \p$. 
Let $\da$ be a maximal abelian subspace of $\p$, and $\bigtriangleup$ be the restricted root systems of $\g$ 
with respect to $\da$. 
Denote by $\g_0$ the centralizer of $\da$ in $\g$ and by $\g_\alpha$ the root space of a root $\alpha$. 
Then $\g$ is decomposed into the direct sum of vector spaces $\g = \g_0 + \sum_{\alpha \in \bigtriangleup} \g_\alpha$. 
Let $\Lambda = \{\alpha_1, \ldots, \alpha_r\}$ be a set of simple roots of $\bigtriangleup$. 
Thus $<\Lambda>_{\R} = \da^*$.  
Denote by $\{H^1, \ldots, H^r\}$ the dual basis of $\Lambda$.  

Let $\Lambda'$ be a proper subset of $\Lambda$. Suppose $\Lambda \setminus \Lambda' = \{\alpha_{i_1}, \ldots, \alpha_{i_k} \}$. 
Put $Z:= H^{i_1} + \cdots + H^{i_k}$. The characteristic element $Z$ defines the subspace 
$\g^0 = \g_0 + \sum_{\alpha(Z)=0}\g_\alpha$ and $\g^k = \sum_{\alpha(Z)=k} \g_\alpha$ for $k \neq 0$. 
Then we obtain the gradation $\g = \sum \g^k$.  

The nonnegative part $\q_{\Lambda'} = \sum_{k \geq 0} \g^k$ gives a parabolic subalgebra, which is equal to  
$\g_0 + \sum_{\alpha \in \bigtriangleup, \alpha(Z) \geq 0}\g_\alpha$. We also have $\q_{\Lambda'}$ $=$  
$\g_0 +  \sum_{\beta \in \bigtriangleup^+ \cup <\wedge'>}\g_\beta$. 
The Langlands decomposition is given by $\q_{\Lambda'} = \m_{\Lambda'} + \da_{\Lambda'} + \n_{\Lambda'}$, 
where the direct summands are defined as follows: 

\noindent 
(1) $\da_{\Lambda'} = <H^{i_1}, \ldots, H^{i_k}>_{\R}$,  \\
\noindent 
(2) $\m_{\Lambda'} = \g^0 \ominus \da_{\Lambda'}$,  \\
\noindent 
(3) $\n_{\Lambda'} = \sum_{k>0} \g^{k}$.  

\medskip 
Then the subalgebra $\s_{\Lambda'} := \da_{\Lambda'} + \n_{\Lambda'}$ of $\q_{\Lambda'}$ is solvable.  
Note that $\n_{\Lambda'} = \sum_{\beta \in \bigtriangleup^+ - <\Lambda'>^+}\g_\beta$.  
In particular when $\Lambda' = \emptyset$, we have $\q_\emptyset = \g_0 + \sum_{\beta \in \bigtriangleup^+}\g_\beta$, which is called a minimal parabolic subalgebra. In this case the characteristic element $Z = H^1 + \cdots +H^r$ gives a Langlands  decomposition 
$\q_{\emptyset} = \m_{\emptyset} + \da_{\emptyset} + \n_{\emptyset}$, where $\da_{\emptyset} = \da$ and 
$\n_{\emptyset} = \sum_{\beta \in \bigtriangleup^+}\g_\beta$. 
We also have the decomposition $\g = \kk + \da + \n_{\emptyset}$ 
which is called the Iwasawa decomposition. Thus the solvable subalgebra $\s_{\emptyset} = \da + \n_{\emptyset}$ of $\q_{\emptyset}$ is same as the solvable  subalgebra of $\g$ associated to the Iwasawa decomposition. 
For any proper subset $\Lambda' \subset \Lambda$ we have $\s_{\Lambda'} \subset \s_{\emptyset}$.

\begin{prop} 
The solvable Lie algebra $\s_\emptyset$ 
admits a flat affine connection $\nabla$. 
The solvable Lie algebra $\s_\wedge'$ is an autoparallel subalgebra of $(\s_\emptyset, \nabla)$. 
\end{prop} 
\begin{proof} 
By definition $\s_\emptyset = \da_\emptyset + \n_\emptyset$.  
The characteristic element $Z = H^1 + \cdots +H^r$ defines the subspace $\g^{k} = \sum_{\alpha(Z)=k}\g_\alpha$ and 
we have $\n_\emptyset  =  \sum_{k>0} \g^{k}$.  Thus $\n_\emptyset$ is graded by positive integers. 
Furthermore $\da_\emptyset$ is abelian and preserves the gradation, i.e.  $[\da_\emptyset, \g^{k}] \subset \g^{k}$. 
It follows that $\s_\emptyset$ admits a flat affine connection $\nabla$.  
Here we recall the construction of $\nabla$ on $\s_\emptyset$.  
If $X, Y \in \s_\emptyset$, then $\nabla$ is given as follows: 
\[ 
\begin{array}{cc|c} 
X & Y  &  \nabla_X Y \\
\hline 
\da & \da & 0 \\
\g^i &  \g^j & \frac{j}{i+j}[X, Y] \\
\da & \n_\emptyset & [X, Y] \\
\n_\emptyset & \da & 0 
\end{array}.  
\] 
Now we consider the subalgebra $\s_{\Lambda'} = \da_{\Lambda'} + \n_{\Lambda'}$. 
The two summands $\da_{\wedge'}$ and $\n_{\wedge'}$ are subalgebras of $\da_\emptyset$ and $\n_\emptyset$ respectively. 
Thus $\s_\wedge'$ is an autoparallel subalgebra of $(\s_\emptyset, \nabla)$ from the construction of $\nabla$.   

\end{proof} 
This proof concerning $\s_\emptyset$ is the same as the one of Theorem 1 in \cite{takemoto-yamaguchi}. 
Indeed when we write $\alpha = \sum_{i=1}^r m_i \alpha_i$, we have $\alpha(Z) = \sum_{i=1}m_i = |\alpha|$.  
Thus $\sum_{|\alpha|=k}\g_\alpha = \g^k$, which also shows that $\n_\emptyset = \sum_{\beta \in \bigtriangleup^+}\g_\beta$ is graded by positive integers. 

\begin{rem} 
The nilpotent part $\n_{\Lambda'}$ of $\s_{\Lambda'}$ also has the gradation $\n_{\Lambda'} = \sum_{k>0}g^k$ defined by the 
characteristic element $Z:= H^{i_1} + \cdots + H^{i_k}$. Hence $\da_{\Lambda'} + \n_{\Lambda'}$ also has the semidirect structure such that 
the adjoint action of $\da_{\Lambda'}$ on $\n_{\Lambda'}$ preserves its gradation. Hence by the same construction we obtain the flat affine connection $\nabla^{\Lambda'}$ on $\s_\Lambda'$. Consequently now we have two flat affine connections on  $\s_\Lambda'$, one is the induced connection $\nabla$ from $(\s_\emptyset, \nabla^\emptyset)$ and the other is $\nabla^{\Lambda'}$. 
These two affine connections are generally different as it is verified by the following examples.  

We consider $\g = \sll(n, \K)$, where $\K = \R$ or $\bH$. In this case a Cartan involution is given by  $\sigma (X) = -{}^t X$ when $\K = \R$ and $\sigma (X) = -{}^t \bar{X}$ when $\K = \bH$. Then the maximal abelian subspace of $\p$ is the set of real diagonals in $\sll(\K, \R)$. Let $\lambda_i$ be a linear function $\da \to \R$ defined by 
$\lambda_i (E_{jj}) = \delta_{ij}$. 
The root system $\bigtriangleup(\g, \da)$ is given by $\{\lambda_i - \lambda_j \ (1 \leq  i \neq j \leq n)\}$. 
The root space $\g_{\lambda_i - \lambda_j}$ of the root $\lambda_i - \lambda_j$ is $\K E_{ij}$.  
Put $\alpha_i = \lambda_i - \lambda_{i+1}$. Then a set of simple roots $\Lambda$ is given by 
$\{\alpha_1, \ldots, \alpha_{n-1}\}$.  
The centralizer $\g_0$ of $\da$ in $\g$ is given by the diagonal part of $\sll(n, \K)$.

Let us consider the case $\g = \sll(4, \R)$. A set of simple roots is given by $\Lambda = \{\alpha_1, \alpha_2, \alpha_3\}$. 
Denote by $\{H^1, H^2, H^3\}$ the dual basis of $\Lambda$. Put $\Lambda' = \{\alpha_3\}$. 
Then we have $\da_{\Lambda'} = \langle H^1,  H^2 \rangle$ and $\n_{\Lambda'} = \g_{\alpha_1} +  \g_{\alpha_2} + \g_{\alpha_1+\alpha_2} + \g_{\alpha_2+\alpha_3} + \g_{\alpha_1+\alpha_2+\alpha_3}$. 
The characteristic element $Z = H^1 + H^2$ corresponding to $\Lambda'$ determines the gradation 
of $\n_{\Lambda'} = \g^1 + \g^2$ given by 
$\g^1 = \g_{\alpha_1} + \g_{\alpha_2} + \g_{\alpha_2+\alpha_3}$ and $\g^2 = \g_{\alpha_1+\alpha_2} + \g_{\alpha_1+\alpha_2+\alpha_3}$.  
Now we compare the two affine connections $\nabla^\emptyset$ and $\nabla^{\wedge'}$ on $\s_{\wedge'}$. 
By the straightforward computations the only difference between two connections are described as follows: 
\[
\begin{array}{ll} 
{\nabla^\emptyset}_{E_{12}} E_{24} = \frac{2}{3}E_{14},  &  {\nabla^{\wedge'}}_{E_{12}} E_{24} = \frac{1}{2}E_{14}, \\[3mm]  
{\nabla^\emptyset}_{E_{24}} E_{12} = \frac{1}{3}(-E_{14}),  & {\nabla^{\wedge'}}_{E_{12}} E_{24} = \frac{1}{2}(-E_{14}).
\end{array}
\]
All the other components has the same values.  
\end{rem} 

Agaoka \cite{agaoka}, Urakawa \cite{urakawa} and Elduque \cite{elduque} proved that simple Lie algebra $\g$ admits a left  invariant projectively flat affine connection iff 
$\g$ is $\sll(n, \R)$ or $\sll(n, \bH)$. Nomizu and Pinkall \cite{nomizu-pinkall} also proved that these Lie algebras are the only simple Lie algebras
admitting biinvariant projectively flat equiaffine connections. 
In fact they showed that 
a Lie algebra with a biinvariant projectively flat affine connection gives rise to 
an associative algebra with identity and vice versa which we now recall.  
Let $A$ be an associative algebra with unit $e$. Let $\tau$ be the linear function on $A$ defined by 
$\tau(u) := \tr(v \in A \mapsto uv \in A)$.  Denote by $\g$ the subspace $\{u \in A \mid \tau(u)=0\}$. We regard $A$ as a Lie algebra with 
the bracket $[u, v] = uv-vu$. Then $\g$ becomes a Lie subalgebra of $A$ such that $A = g \oplus \R e$. 
We define the left invariant affine connection $\nabla$ on $\g$ by $\nabla_XY = XY - \frac{\tau(XY)}{n+1}e$. 
Then $\nabla$ is verified to be biinvariant.  
Moreover $Ric(X, Y) = -\frac{n-1}{n+1} \tau(XY)$ and thus Ricci tensor is symmetric and Weyl's projective curvature vanishes.  
Consequently we obtain a left invariant projectively flat affine connection $\nabla$ on $\g$. 

In particular $\gl(n, \R)$ and $\gl(n, \bH)$ are associative algebras with unit, hence we obtain the Lie algebras equipped with 
left invariant projectively flat affine connections $(\sll(n, \R), \nabla)$ and $(\sll(n, \bH), \nabla)$.  
The function $\tau$ is given by $\tau(X) = n\tr X$ and $\tau(X) = 4n \mathrm{Re}\tr X$  
respectively for $\gl(n, \R)$ and  $\gl(n, \bH)$. 
Thus we have 
\begin{eqnarray*} 
\nabla_XY &=& XY - \frac{\tr XY}{n}I_n,  \quad  \quad \ \gamma(X, Y) = -\frac{\tr XY}{n}   \quad \quad  \mbox{on} \quad \sll(n, \R),   \\ 
\nabla_XY &=& XY - \frac{\mathrm{Re} \tr XY}{n}I_n, \quad  \gamma(X, Y) = - \frac{\mathrm{Re} \tr XY}{n} \quad  \mbox{on} \quad \sll(n, \bH).  
\end{eqnarray*}  

\begin{prop}\label{autoparallel parabolic in sl}
Parabolic subalgebras of $(\sll(n, \R), \nabla)$ and $(\sll(n, \bH), \nabla)$ are autoparallel.  
\end{prop} 
\begin{proof} 
Let $\q_{\Lambda'}$ be a parabolic subalgebra of $\sll(n, \K)$ where $\K = \R$ or $\bH$. 
Denote by $Z$ the characteristic element determined by $\Lambda'$. 
We show that $\q_{\Lambda'}$ is closed under the multiplication $\nabla$. 
We consider the root space decomposition $\sll(n, \K) = \g_0 + \sum_{\alpha \in \bigtriangleup}\g_\alpha$. 
Then $\g_0$ is the diagonal of $\sll(n, \K)$ and  $\g_\alpha = \K E_{ij}$ for $\alpha = \lambda_i - \lambda_j$.  
Hence from the definition of $\nabla$ on $\sll(n, \K)$ for $\alpha \in \bigtriangleup^+ \cup <\Lambda'>$ 
obviously we have $\nabla_{\g_0} \g_0 \subset \g_0$, \  $\nabla_{\g_\alpha} \mathfrak{g}_0 \subset \g_\alpha$, \ 
$\nabla_{\mathfrak{g}_0} \g_\alpha \subset \g_\alpha$. 
Therefore it is enough to prove $\nabla_{\g_\alpha} \g_\beta \subset 
\sum_{\gamma \in \bigtriangleup^+ \cup <\Lambda'>}\g_\gamma$ for $\alpha, \beta \in \bigtriangleup^+ \cup <\Lambda'>$. 
To prove this we observe that $\nabla_{\g_\alpha} \g_\beta \subset \g_{\alpha +\beta}$ for $\alpha, \beta \in \bigtriangleup$. 

Firstly we consider the case $\beta = -\alpha$. 
Since $\nabla_{E_{ij}}E_{ji} = E_{ii} - \frac{1}{n}I_n$, we have $\nabla_{\K E_{ij}}\K E_{ji} \subset \g_0$. 
Thus $\nabla_{\g_\alpha} \g_{-\alpha} \subset \g_0$. 

Secondly suppose $\beta \neq - \alpha$. Then 
we have $\g_\alpha \g_\beta \subset \g_{\alpha+\beta}$. This yields that $\nabla_{\g_\alpha} \g_\beta \subset \g_{\alpha +\beta}$. 
Therefore $\alpha + \beta \notin \bigtriangleup$ iff $\nabla_{\g_\alpha} \g_\beta = 0$ and $\nabla_{\g_\beta} \g_\alpha = 0$.   
On the other hand if $\alpha + \beta \in \bigtriangleup$ and moreover $\alpha(Z) \geq 0$, $\beta(Z) \geq 0$, then we have 
$\nabla_{\g_\alpha} \g_\beta \subset \sum_{\gamma \in \bigtriangleup, \gamma(Z) \geq 0}\g_\gamma$.  
Consequently $\q_{\Lambda'}$ is closed under the multiplication $\nabla$.  
\end{proof}

\begin{prop}\label{autoparallel solvable in sl}
The solvable subalgebra $\s_{\Lambda'}$ is autoparallel in $(\sll(n, \K), \nabla)$. 
\end{prop}
\begin{proof} 
Obviously we have $\nabla_{\da_{\Lambda'}}\n_{\Lambda'} \subset \n_{\Lambda'}$ and $\nabla_{\n_{\Lambda'}}\da_{\Lambda'} \subset \n_{\Lambda'}$.   
Now assume that $\g_\alpha, \g_\beta \subset \n_{\Lambda'}$. Then $\alpha + \beta \neq 0$ and $\alpha +\beta (Z) > 0$. 
As we have seen in the proof of Proposition \ref{autoparallel parabolic in sl}, $\nabla_{\g_\alpha} \g_\beta \subset 
\g_{\alpha +\beta}$. 
Thus if $\alpha + \beta \notin \bigtriangleup$, then $\nabla_{\g_\alpha} \g_\beta =0$. If $\alpha + \beta \in \bigtriangleup$, then 
$\nabla_{\g_\alpha} \g_\beta \subset \g_{\alpha +\beta} \subset \n_{\Lambda'}$. 

Finally we show $\nabla_{\da_{\Lambda'}}\da_{\Lambda'} \subset \da_{\Lambda'}$. 
The dual basis $\{H^1, \ldots, H^{n-1}\}$ of $\{\alpha_1, \ldots, \alpha_{n-1}\}$ is given by 
\[H^i = {\small \frac{1}{n} 
\begin{pmatrix}
n-i  &           &     &      &           & \\
      & \ddots &     &       &          & \\
      &           & n-i &      &          &  \\
      &           &      & -i  &           &  \\
      &           &      &     & \ddots & \\
      &           &      &     &           & -i 
\end{pmatrix}}, \] 
where the first $i$ components of the diagonal equal $n-i$ and the latter $n-i$ components equal $-i$. 
Then from the direct calculation we have $\nabla_{H^i}H^j = \frac{1}{n}[(n-j)H^i - i H^j]$.    
Therefore $\s_{\Lambda'}$ is closed under the multiplication $\nabla$. 
\end{proof}

We showed that a parabolic subalgebra $\q_{\Lambda'}$ and a solvable subalgebra $\s_{\Lambda'}$ of $(\sll(n, \R), \nabla)$ and 
$(\sll(n, \bH), \nabla)$ are autoparallel, hence on which projectively flat affine connections $\nabla$ are induced. However 
the induced connections $\nabla$ on $\q_{\Lambda'}$ and $\s_{\Lambda'}$ are not affinely flat. Indeed 
$\q_{\Lambda'} \supset \da = <H^1, \ldots, H^{n-1}>_{\R}$ and 
the normalized Ricci tensor 
$\gamma^{\q_{\Lambda'}}$ of $(\q_{\Lambda'}, \nabla)$ satisfies $\gamma^{\q_{\Lambda'}}(H^i, H^i) = \gamma (H^i, H^i) = -\frac{(n-i)i}{n^2}$ for $1 \leq  i \leq n-1$. Thus $(\q_{\Lambda'}, \nabla)$ is not affinely flat.  
On the other hand $\s_{\Lambda'}$ may not contain the whole space $\da$, but still contains at lease one $H^i$. Then  
$\gamma^{\s_{\Lambda'}}(H^i, H^i) = \gamma (H^i, H^i)$ and thus $(\s_{\Lambda'}, \nabla)$ is also not affinely flat. 
However in the following we prove $(\s_{\Lambda'}, \nabla)$  is projectively equivalent to a flat affine connection. 
For this purpose we introduce invariants. 

Two torsion-free affine connections $\nabla$ and $\nabla'$ on $M$ is said to be projectively equivalent if there exists a 1-from $\lambda$ on $M$ such  that $\nabla_X Y - \nabla'_X Y$ $=$ $\lambda(X) Y + \lambda(Y) X$. 
If both $\nabla$ and $\nabla'$ are left invariant affine connections on a Lie group $L$, then $\lambda$ becomes also left invariant. 
Let $f$ and $f'$ be the linear map $\l \to \sll(n+1, \R)$ induced by $\nabla$ and $\nabla'$ respectively. Then the projective equivalence  relation is interpreted as follows: 
$\nabla$ is projectively equivalent to $\nabla'$ iff there exists $\xi \in {\R^n}^*$ such that 
$f' = 
\begin{pmatrix}
I_m  & 0 \\
-\xi & 1
\end{pmatrix}^{-1} f \ 
\begin{pmatrix}
I_m &  0 \\
-\xi & 1
\end{pmatrix}$. 
Indeed left invariant 1-form $\lambda$ and $\xi$ is related by $\xi = (\lambda(X_1), \ldots, \lambda(X_n))$.  

Let $\nabla$ be a left invariant projectively flat affine connection on $L$ and $f$ a corresponding (P)-homomorphism.  
There is a useful tool called invariants to determine a projective equivalence class of $\nabla$ contains a flat affine connection.  Let $\{X_1, \ldots, X_n\}$ be a basis of $\l$. Then the invariant $\phi_f: \R^{n+1} \to \R$ corresponding to $f$ is defined by  
$\phi_f(v) = \det(f(X_1)v, \ldots, f(X_n)v, v)$. 
Then the projective equivalence class $[\nabla]$ contains affinely flat connection iff 
the invariant $\varphi$ induced by $\nabla$ possesses a real linear factor involving $x_{n+1}$, i.e. 
$\varphi(v) = (a_1 x_1+ a_2 x_2 + \cdots + a_n x_{n} + a_{n+1} x_{n+1}) \psi(v)$ for some $(a_1, \ldots, a_{n+1})^t \in {\R^{n+1}}$ (see \cite{agaoka-kato}).   

Assume $f: \l \to \gl(n+1, \R)$ is a (P)-homomorphism corresponding to $\nabla$ on $\l$.  
Suppose that $\varphi_f(v) = (a_1 x_1+ a_2 x_2 + \cdots + a_n x_{n} + a_{n+1} x_{n+1}) \psi(v)$ 
for some $(a_1, \ldots, a_{n+1})^t \in {\R^{n+1}}$.  
Put $\xi = \frac{1}{a_{n+1}}(a_1, \ldots, a_{n})$,  \ 
$Q = 
\begin{pmatrix}
I_m  & 0 \\
-\xi & 1
\end{pmatrix}$ 
and $f' = Q^{-1} f Q$.  
Then we have 
\begin{eqnarray*}
\varphi_{f'}(v) &=& 
\varphi_f (
Q  v) \\
&=& 
(a_{n+1} x_{n+1})
\psi_f (Q v). 
\end{eqnarray*} 
The invariant $\varphi_{f'}(v)$ possesses a linear factor $x_{n+1}$,  it follows that $f'_1=0$ (cf. \cite{agaoka-kato}). 
Hence   
$f' = Q^{-1} f Q$ 
gives a (P)-homomorphism corresponding to a flat affine connection $\nabla'$, which is projectively equivalent to $\nabla$.  
Now we shall prove the following: 

\begin{prop}
The induced affine connection $\nabla$ on $\s_{\Lambda'}$ is projectively equivalent to a flat affine connection. 
\end{prop}
\begin{proof}
By definition $\s_{\Lambda'} = \da_{\Lambda'} + \n_{\Lambda'}$, where $\da_{\Lambda'} = <H^{i_1}, \ldots, H^{i_k}>_{\R}$ and    
$\n_{\Lambda'} = \sum_{\bigtriangleup^+ - <\Lambda'>^+}$. Hence 
always we have $\n_{\Lambda'} \supset \g_{\alpha_1 + \cdots + \alpha_{n-1}} = \K E_{1n}$. 
According to the definition of $\nabla$ on $\sll(n, \K)$ we have $\nabla_{E_{1n}} H^{i_j} = E_{in} H^{i_j} = \frac{-i_j}{n} E_{1n}$, and 
$\nabla_{E_{1n}} \K E_{kl} = 0$ for $E_{kl} \in \n_{\Lambda'}$.  Denote by $m$ the dimension of $\s_{\Lambda'}$. 
It follows that a Lie algebra representation $f: \s_{\Lambda'} \to \gl(m+1, \R)$ constructed from $\nabla$ on $\s_{\Lambda'}$ 
is of the form 
\[
f(E_{1n}) = 
\begin{pmatrix}
&&&&&&  & 0 \\
&&& \mbox{\smash{\huge 0}} &&& & \vdots \\
&&& &&&  & 0 \\
-\frac{i_1}{n} & -\frac{i_2}{n} & \cdots & -\frac{i_k}{n} & 0 & \cdots & 0 & 1 \\
0 & 0 & \cdots & 0 & 0 & \cdots & 0 & 0 
\end{pmatrix}. \] 
Therefore the invariant $\varphi_f : \R^{m+1} \to \R$ induced from $f$ is calculated as follows: 
\begin{eqnarray*} 
\varphi_f(v) &=& \det (f(H^{i_1})v, f(H^{i_2})v, \ldots, f(H^{i_k})v, \ldots, f(E_{1n})v, v) \\  
&=& 
\det 
\begin{pmatrix}
*  & 0 & x_1 \\ 
\vdots  & \vdots & \vdots \\ 
*  & 0 & x_{m-1} \\ 
*  &   - \frac{i_1}{n}x_1 - \frac{i_2}{n}x_2 - \cdots - \frac{i_k}{n}x_k + x_{m+1} & x_m \\
*  &  0 & x_{m+1} 
\end{pmatrix} \\ 
&=& -\frac{1}{n}(i_1 x_1 + i_2 x_2 + \cdots + i_k x_k - n x_{m+1}) \psi (v). 
\end{eqnarray*}   
Denote by $\xi$ a row vector $-\frac{1}{n}(i_1, i_2, \ldots, i_k, 0, \ldots, 0)$ of the length $m$.  
Put ${}^tv := (x_1,  \cdots, x_m, x_{m+1})$ and $Q := 
\begin{pmatrix}
I_m  & 0 \\
-\xi & 1
\end{pmatrix}$. 
Then 
\[
Q v 
= 
\begin{pmatrix}
x_1  \\ 
\vdots \\ 
x_m \\ 
\frac{i_1}{n}x_1 + \frac{i_2}{n}x_2 \cdots + \frac{i_k}{n}x_k + x_{m+1} 
\end{pmatrix}. \] 
Thus we have 
\[
\varphi_f \left( 
Q v \right) = 
x_{m+1} \cdot \psi \left(
Q v\right). 
\] 
Since $\varphi_{Q^{-1} \bar{f}Q}(v) = \varphi_{Q^{-1} fQ}(v) = \varphi_f(Q v)$, 
the (P)-homomorphism 
$Q^{-1} \bar{f}Q$ 
of $\s_{\Lambda'}$ is corresponding to a flat affine connection.  
This proves the proposition.   
\end{proof}

\section{Affinely flat parabolic subgroups} 
Let us recall that the set $\Lambda = \{\alpha_1, \ldots, \alpha_{n-1}\}$ gives a set of simple roots of $\bigtriangleup (\sll(n, \R), \da)$. 
Let  $\Lambda' = \{\alpha_{i_1}, \alpha_{i_2}, \ldots, \alpha_{i_m}\}$ be a proper subset of $\Lambda$. 
Without loss of generality we can assume $i_1 < i_2 < \cdots < i_{m}$. 
To begin with we compute the induced affine connection $\nabla$ on $\q_{\Lambda'}$.  Recall that the Ricci symmetric  connection $\nabla$ on $\sll(n, \R)$ is given by $\nabla_X Y = X Y -\frac{\tr XY}{n}I_n$ for $X, Y \in \sll(n, \R)$. 
The straightforward computation yields the following:  
 
\begin{eqnarray*}
\nabla_{H^k}H^l &=&  \frac{1}{n}[(n-l)H^k - k H^l] \ \mbox{for} \ k \leq l  \\  
\nabla_{H^k} E_{i j} &=&   
\left\{ 
\begin{array}{l}
\frac{n-k}{n}E_{i j} \ \mbox{for} \ i \leq k \\[2mm]
\frac{-k}{n}E_{i j} \ \mbox{for} \ i \geq k+1
\end{array}\right. \\
\nabla_{E_{i j}}H^k &=& 
\left\{
\begin{array}{l}
\frac{n-k}{n} E_{i j} \ \mbox{for} \ j \leq k \\[2mm]
\frac{-k}{n} E_{i j} \ \mbox{for} \ j \geq k+1 
\end{array}\right.  \\
\nabla_{E_{i j}}E_{k l} &=& \delta_{j k} E_{i l} \ \mbox{for} \ i \neq l  \\
\nabla_{E_{i j}}E_{j i} &=& - H^{i-1} + H^i  \ (H^0 = H^n = 0)\\ 
\gamma(H^i, H^j) &=& -\frac{i(n-j)}{n^2} \ (i \leq j) \\ 
\gamma(H^k, E_{i j}) &=& 0  \\ 
\gamma(E_{i j}, E_{k l}) &=& - \delta_{j k} \delta_{i l} \frac{1}{n}.  
\end{eqnarray*}
By using these data we can prove the following:  
\begin{prop}\label{main prop1}  
Assume that $i_1 = 1$,  $i_m = n-1$ and  
$|i_r - i_{r+1}| \leq 2$ for $1 \leq r \leq m - 1$. 
Then the induced affine connection $\nabla$ on $\q_{\Lambda'}$ is not projectively equivalent to any flat affine connection. 
\end{prop}
\begin{proof} 
From assumption we can choose a basis of $\q_{\Lambda}$ as 
\[\q_{\Lambda} = <H^1, H^2, \ldots, H^{n-1} \mid E_{12}, E_{21}, \ldots, E_{i_r i_{r+1}}, E_{i_{r+1} i_r}, \ldots, E_{n-1 n}, E_{n n-1} 
\mid \ldots >. \] 
The first part is the basis of $\g_0 = \da$, the second part is the basis of 
$\sum_{\alpha \in \Lambda'} \g_\alpha \oplus \g_{-\alpha}$, and the third part is the remaining basis. 
Let $f: \q_{\Lambda'} \to \gl(\dim \q_{\Lambda'} +1, \R) $ be a representation of the form \ref{form of rep} corresponding to $\nabla$. 
Let us describe $f$ by matrices with respect to the decomposition 
$\q_{\Lambda'} \oplus <\R> = \g_0 + \sum_{\beta \in \bigtriangleup^+ \cup <\Lambda'>}\g_\beta \oplus <\R>$. 
\begin{eqnarray*} 
 && \ f(H^k)  \\
&=& \hspace{-3mm} \frac{1}{n} 
\left({\arraycolsep=0.1mm 
\begin{array}{ccccccccc|ccc|c} 
n-k&&&&&&&&    &&  &  & 0 \\
&n-k&&&&&&&    &&  &  &  0 \\
&&\ddots&&&&&&    &&  &  & \vdots \\ 
&&& n-k &&&&&    &&  &  & 0 \\
-1&-2&\cdots &-(k-1) &n-2k &n-(k+1) &n-(k+2) &\cdots & 1   && \mbox{\smash{\huge 0}} &  & n \\
&&&&&-k &&&    &&  &  & 0 \\
&&&&&&-k &&    &&  &  & 0 \\
&&&&&&&\ddots &    &&  &  & \vdots \\
&&&&&&&& \! \! -k    &&  &  & 0 \\
\hline
&&&&&&&&    &&  &  &  \\
&&&&\mbox{\smash{\huge 0}} &&&&    && \mbox{\smash{\huge $\ast$}} &  & \mbox{\smash{\huge 0}}  \\
&&&&&&&&    &&  &  &  \\
\hline 
\frac{1}{n}(n-k) &\frac{2}{n}(n-k) &\cdots &\frac{k-1}{n}(n-k) &\frac{k}{n}(n-k) &\frac{k}{n}\{n-(k+1)\}&\frac{k}{n}\{n-(k+2)\}&\cdots &\frac{k}{n}    & &0  &  &0  \\
\end{array}}
\right). \hspace{8mm}
\end{eqnarray*}

\begin{eqnarray*}
f(E_{i j}) = 
\frac{1}{n}
\left(
\begin{array}{c|ccccccc|c}
& &&& 0 &&& &    \\
& &&& \vdots &&& &    \\
& &&& 0 &&& &    \\
& &&& \! \! -n &&& &    \\
\mbox{\smash{\huge 0}} &    &&& n &&& &  0 \\    
& &&& 0 &&& &    \\
& &&& \vdots &&& &    \\
& &&& 0 &&& &    \\
\hline
&   &&& 0 &&& & 0 \\
&   &&& \vdots &&& & \vdots \\
&   &&& 0 &&& & 0 \\
\mbox{\smash{\huge $\ast$}} & & \mbox{\smash{\huge $\ast$}}  &&  0 && \mbox{\smash{\huge $\ast$}} & & n \\  
&   &&& 0  &&& & 0 \\
&   &&& \vdots &&& & \vdots \\
&   &&& 0 &&& & 0 \\
\hline
0 &  0& \cdots & 0 & 1 & 0 & \cdots & 0 & 0 
\end{array}
\right), \ 
f(E_{k l}) = 
\frac{1}{n}
\left(
\begin{array}{ccc|ccc|c}
&&& && &    \\
&\mbox{\smash{\huge 0}}&& &\mbox{\smash{\huge 0}}& & 0  \\
&&& && &    \\
\hline
&&& && & 0  \\
&&& && & \vdots  \\
&&& && &  0 \\
&\mbox{\smash{\huge $\ast$}}&& &\mbox{\smash{\huge $\ast$}}& & n   \\
&&& && & 0  \\
&&& && & \vdots  \\
&&& && &  0 \\
\hline
&0&& &0& & 0   \\ 
\end{array}
\right)
\end{eqnarray*}

In the above $E_{i j}$ is belonging to $\sum_{\beta \in <\Lambda'>}\g_\beta$ and $E_{k l}$ is not belonging to 
$\sum_{\beta \in <\Lambda'>}\g_\beta$. Furthermore $f(E_{i j})$ has the 9 block decomposition, and the $(1, 2)$-part of 
$f(E_{i j})$ is expressing $\nabla_{E_{i j}}E_{j i} = -H^{i-1} + H^i$ and the $(2, 3)$-part is corresponding to $E_{i j}$. 
The $(3, 2)$-part is expressing $-\gamma(E_{i j}, E_{j i}) =  \frac{1}{n}$.  Note that $f(E_{i j})E_{i j} = 0$. 
    
Denote by $\Xi$ a row vector 
\[(\xi_1 \ \xi_2 \ \cdots \ \xi_{n-1} \mid \zeta_{12} \ \zeta_{21} \cdots \zeta_{i_r i_{r+1}} \ \zeta_{i_{r+1} i_r} \cdots  
\xi_{n-1 n} \ \xi_{n n-1} \mid \cdots )\]  
of the length $m = \dim \q_{\Lambda'}$, whose variables are corresponding to the above basis.  
Denote by $P$ a $m+1 \times m+1$ matrix    
$
\begin{pmatrix} 
I_m & 0 \\
- \Xi & 1 
\end{pmatrix}. 
$ 
The $\g_1$ component of $P^{-1} f(E_{12}) P$ is equal to 
$(\cdots \mid 0 \ n\xi_1+1 \ \cdots ) - 
n\zeta_{12}(\xi_1 \ \xi_2 \ \cdots \ \xi_{n-1} \mid \zeta_{12} \ \zeta_{21} \cdots )$.  
Thus if a representation $P^{-1} f P$ is corresponding to a flat affine connection, then we have  
$\zeta_{12} = 0$ and $\xi_1 = -\frac{1}{n}$.  
Likewise from the computation of $P^{-1} f(E_{21}) P$ we obtain $\zeta_{21} = 0$ and $-n \xi_1 + n \xi_2 + 1 = 0$ as the 
necessary  condition of $P^{-1} f P$ being affine flat.

Generally 
the $\g_1$ component of $P^{-1} f(E_{i_r i_{r+1}}) P$ is equal to 
{\small 
\begin{eqnarray*} 
&&  (\cdots \mid \cdots 0  -n \xi_{i_r-1} + n\xi_{i_r} + 1 \cdots ) \\  
&& - \ n\zeta_{i_r i_{r+1}}(\xi_1 \ \xi_2  \cdots  \xi_{n-1} \mid \cdots  \zeta_{i_r i_{r+1}} \  \zeta_{i_{r+1} i_r} \cdots ). 
\end{eqnarray*}
}
Thus we obtain  $\zeta_{i_r i_{r+1}} = 0$ and $- n\xi_{i_r-1}+n\xi_{i_r}+1 = 0$. 
From the computation of $P^{-1} f(E_{i_{r+1} i_r}) P$ we obtain  
$\zeta_{i_{r+1} i_r} = 0$ and $- n\xi_{i_r}+n\xi_{i_r+1}+1 = 0$.  
Combining these these equations yields the following:  
\[
\begin{array}{l}
\ \  n\xi_1 +1 = 0 \\ 
-n \xi_1 + n \xi_2 + 1=0 \\ 
-n \xi_2 + n \xi_3 + 1=0  \\ 
\quad \quad \quad \vdots  \\ 
-n \xi_{n-2} + n \xi_{n-1} +1=0  \\ 
-n \xi_{n-1} +1 = 0. 
\end{array} 
\]
These equations have no common solutions.  
Indeed the first $(n-1)$ equations have the common solution 
$\xi_1 = -\frac{1}{n}, \xi_2 = -\frac{2}{n}, \ldots, \xi_{n-1} = -\frac{n-1}{n}$, 
however which contradicts the last equation.   

It follows that the induced affine connection $\nabla$ on $\q_{\Lambda'}$ is not projectively equivalent to any flat affine connection.  
\end{proof}
The converse of Proposition \ref{main prop1} is also true. 
\begin{prop}\label{main prop2} 
Assume that $\Lambda' \varsubsetneq \Lambda$ does not satisfy $i_1 = 1$,  $i_m = n-1$ and 
$|i_r - i_{r+1}| \leq 2$ for $1 \leq r \leq m - 1$. 
Then the induced affine connection $\nabla$ on $\q_{\Lambda'}$ is projectively equivalent to a flat affine connection. 
\end{prop}
\begin{proof} 
Let us choose the basis of $\q_{\Lambda'} \oplus \R$ as 
\[\q_{\Lambda'} \oplus <w> = \{H^1, H^2, \ldots, H^{n-1} \mid E_{12}, E_{21}, \ldots, E_{i_r i_{r+1}}, E_{i_{r+1} i_r}, \ldots, E_{n-1 n}, E_{n n-1} 
\mid \ldots \mid w \}.\] 
With respect to this basis we express an element $v$ of $\R^{\dim \q_{\Lambda'} + 1}$ as 
$(a_1, \ldots, a_{n-1} \mid b_{i j} \mid z)$.  
By using this basis we define the matrix $f_{v}$ to be    
\begin{eqnarray*} 
(f(H^1)v, f(H^2)v, \cdots, f(H^{n-1})v, \mid f(E_{1 2})v, f(E_{2 1})v, \ldots, f(E_{n-1 n})v,  f(E_{n n-1})v  \mid \cdots \mid v). 
\end{eqnarray*}  
Denote by $\varphi_f$ the invariant induced by the representation $f$.  
Then we have $\varphi_f(v)  =   \det f_v$.  
The column vector of $f_v$ corresponding to $H^k$ is given by
\begin{eqnarray*} 
\frac{1}{n}\left(
\begin{array}{c} 
  (n-k)a_1  \\   
   (n-k)a_2  \\ 
  (n-k)a_3  \\ 
   \vdots   \\  
   (n-k)a_{k-1}        \\
  - a_1-2 a_2 \cdots - (k-1) a_{k-1} + (n-2k) a_k + \{n-(k+1)\}a_{k+1} + \cdots +a_{n-1} + n z    \\ 
  -k a_{k+1}   \\
  \vdots   \\
  -k a_{n-2}   \\ 
  -k a_{n-1} \\ 
\hline
  *  \\
  *  \\
  *   \\ 
\hline 
 \frac{1}{n}(n-k)a_1 + \cdots + 
\frac{k-1}{n}(n-k) a_{k-1} + \frac{k}{n}(n-k) a_{k} + \frac{k}{n}\{n-(k+1)\} a_{k+1} + \cdots + \frac{k}{n} a_{n-1}  \\ 
\end{array}
\right). 
\end{eqnarray*}
The column vectors of $f_v$ corresponding to $E_{i j}$ is given by  
\[
\frac{1}{n}\left( 
\begin{array}{c} 
 0    \\   
 \vdots  \\ 
 0    \\  
  -n b_{j i}  H^{i-1}      \\ 
   n b_{j i}  H^i  \\ 
0     \\ 
 \vdots  \\ 
 0  \\ 
\hline 
 \\
\mbox{\smash{\huge $\ast$}}   \\ 
 \\
\hline 
 b_{j i} 
\end{array}
\right). 
\] 
Note that $H^0 = H^{n} = 0$. For example $E_{n n-1}$-th column of $f_v$ is 
${}^t(0, \ldots, 0 -n b_{n-1 n} \mid \ast \mid b_{n-1 n})$. 
This case appears only when $i_m = n-1$.  
 
We set $i_l = n+1$ if $i_m \neq n-1$,  and put $\xi_r = -\frac{r}{n}$ for $1 \leq r \leq n-1$. 
If $i_m = n-1$, then define  
$i_l$ to be the minimum number of the set $\{i_1, \cdots, i_{j}\}$ such that we have $|i_s - i_{s+1}| \leq 2$ for any $i_s \geq i_l$. 
Set $\xi_r = -\frac{r}{n}$ for $1 \leq  r \leq i_l - 2$, 
and $\xi_r = \frac{n-r}{n}$ for $r \geq i_l - 1$. 
Denote by $\Xi$ the row vector 
$(\xi_1 \ \cdots \ \xi_{n-1} \ 0 \cdots 0)$ of the length $\dim \q_{\Lambda'}$.  
Denote by $P$ the matrix 
$
\begin{pmatrix} 
I_m  &  0 \\
-\Xi & 1
\end{pmatrix}
$
 and by $f'$ the representation 
$P^{-1} f P$.  
Put $z' = - \xi_1 a_1 + \cdots - \xi_{n-1} a_{n-1} + z$. 
Since we have $\varphi_{f'}(v) = \varphi_{f}(P v)$,  
the invariant $\varphi_{f'}(v)$ is obtained by replacing $z'$ with $z$ appearing in $\varphi_f(v)$. 

Let us take vectors $X$, $Y$ from the basis of $\q_{\Lambda'} \oplus <w>$.  
Denote by $\varphi(X, Y)$ the $(X, Y)$ component of the matrix $f_{P v}$.   
Likewise denote by $\varphi(X)$ the row vector of the matrix corresponding to $X$.  
Then even if we replace the last row $\varphi(w)$ of the matrix $f_{P v}$ with the row 
$\xi_1 \varphi(H^1) + \cdots + \xi_{n-1} \varphi(H^{n-1}) + \varphi(w)$, the determinant does not change.  
Now we compute the row vector $\xi_1 \varphi(H^1) + \cdots + \xi_{n-1} \varphi(H^{n-1}) + \varphi(w)$. 
By definition of $\Xi$ we have 
$\xi_1 \varphi(H^1, E_{i j}) + \cdots + \xi_{n-1} \varphi(H^{n-1}, E_{i j}) + \varphi(w, E_{i j}) = 0$ for arbitrary vector $E_{i j} 
\in \q_{\Lambda'}$.  
From the definition of the matrix $P$ and the value of $\xi$ we have 
$\xi_1 \varphi(H^1, w) + \cdots + \xi_{n-1} \varphi(H^{n-1}, w) + \varphi(w, w) = z$. 

Finally we compute the $k$-th column $f(H^k) P v$.   
Suppose $k \leq i_l - 2$.  
Then $\varphi(H^{k}, H^k) = (n-k)a_k + \sum_{k+1\leq  r \leq i_l -2}n a_r + n z$. 
Hence we have  
\begin{eqnarray*} 
&& \xi_1 \varphi(H^1, H^k) + \cdots + \xi_{n-1} \varphi(H^{n-1}, H^k)   \\ 
&=& \hspace{-3mm} \sum_{r' \leq k-1} - \frac{r'}{n}(n-k) a_{r'} 
 - \frac{k}{n}[(n-k)a_k + \! \! \sum_{k+1\leq  r \leq i_l -2}\! \! \! n a_r + n z] \\
&&{} + \sum_{k+1 \leq r \leq i_l-2} \frac{r}{n} k a_r + \sum_{i_l -1 \leq  s} -\frac{n-s}{n} k a_s \hspace{1cm} \\
&=& - \varphi(w, H^k) -k z.  \hspace{15mm}
\end{eqnarray*}

Now suppose $k \geq i_l - 1$. 
Then $\varphi(H^{k}, H^k) = \sum_{i_l-1 \leq  r' \leq k-1} -n a_{r'} + (-k) a_k + n z$. 
Hence we have 
\begin{eqnarray*}
&& \xi_1 \varphi(H^1, H^k) + \cdots + \xi_{n-1} \varphi(H^{n-1}, H^k)  \hspace{3cm}  \\
&=& \sum_{r \leq i_l -2} -\frac{r}{n}(n-k) a_r + \sum_{i_l-1 \leq  r' \leq k-1} \frac{n-r'}{n} (n-k)a_{r'} \\
&& {} + \frac{n-k}{n}[\sum_{i_l-1 \leq  r' \leq k-1} -n a_{r'} + (-k) a_k + n z]  +  \sum_{s \geq k+1} -\frac{n-s}{n}k a_s  \\
&=& -  \varphi(w, H^k) + (n-k) z. 
\end{eqnarray*}

It follows that any component of row vector 
$\xi_1 \varphi(H^1) + \cdots + \xi_{n-1} \varphi(H^{n-1}) + \varphi(w)$ of the matrix $f_{P v}$ 
has only $0$ or $z$ multiplied by a scalar. 
Therefore the invariant $\varphi_{f'}(v)$ possesses a linear factor $z$, which implies that affine connection $\nabla$ 
on $\q_{\Lambda'}$ is projectively equivalent to a flat affine connection.       
\end{proof}
Combining Propositions \ref{main prop1} and \ref{main prop2} we obtain the following: 
Let  $\Lambda' = \{\alpha_{i_1}, \alpha_{i_2}, \ldots, \alpha_{i_j}\}$ be a proper subset of $\Lambda$. Assume that  $i_1 < i_2 < \cdots < i_{j}$.  
\begin{thm}\label{main thm 1}
The induced affine connection $\nabla$ on $\q_{\Lambda'}$ is not projectively equivalent to any flat affine connection 
iff  we have $i_1 = 1$,  $i_j = n-1$ and $|i_r - i_{r+1}| \leq 2$ for $1 \leq r \leq j - 1$.  
\end{thm}
To illustrate Theorem \ref{main thm 1} we consider $\sll(6, \R)$. Then  $\Lambda = \{\alpha_1, \ldots, \alpha_5\}$ can be expressed by the dynkin diagram  
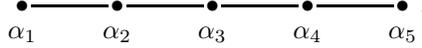
\[
\xygraph{
\bullet ([]!{+(0,-.3)} {\alpha_1}) - [r]
\bullet ([]!{+(0,-.3)} {\alpha_2}) - [r]
\bullet ([]!{+(0,-.3)} {\alpha_3}) - [r]
\bullet ([]!{+(0,-.3)} {\alpha_4}) - [r]
\bullet ([]!{+(0,-.3)} {\alpha_5})}. 
\] 
All parabolic subalgebras of $\sll(6, \R)$ appearing in the theorem is exhausted by the following: 
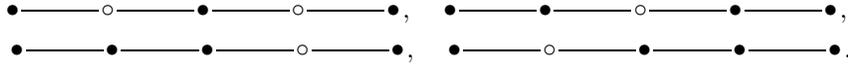
\begin{eqnarray*}
\xygraph{
\bullet ([]!{+(0,-.3)} ) - [r]
\circ ([]!{+(0,-.3)} ) - [r]
\bullet ([]!{+(0,-.3)} ) - [r]
\circ ([]!{+(0,-.3)} ) - [r]
\bullet ([]!{+(0,-.3)} )}, 
\quad 
\xygraph{
\bullet ([]!{+(0,-.3)} ) - [r]
\bullet ([]!{+(0,-.3)} ) - [r]
\circ ([]!{+(0,-.3)} ) - [r]
\bullet ([]!{+(0,-.3)} ) - [r]
\bullet ([]!{+(0,-.3)} )},  
\\
\xygraph{
\bullet ([]!{+(0,-.3)} ) - [r]
\bullet ([]!{+(0,-.3)} ) - [r]
\bullet ([]!{+(0,-.3)} ) - [r]
\circ ([]!{+(0,-.3)} ) - [r]
\bullet ([]!{+(0,-.3)} )}, 
\quad 
\xygraph{
\bullet ([]!{+(0,-.3)} ) - [r]
\circ ([]!{+(0,-.3)} ) - [r]
\bullet ([]!{+(0,-.3)} ) - [r]
\bullet ([]!{+(0,-.3)} ) - [r]
\bullet ([]!{+(0,-.3)} )}. 
\end{eqnarray*}
The first diagram is corresponding to the subset $\Lambda' = \{\alpha_1, \alpha_3, \alpha_5\}$. 
For other parabolic subalgebra $\q_{\Lambda'}$ such as $\Lambda' = \{\alpha_1\}$, the induced affine connection is projectively equivalent to a flat affine connection.   

On the other hand concerning $\sll(n, \bH)$ we have the following: 
\begin{thm} 
The induced affine connection on any parabolic subalgebra of $\sll(n, \bH)$ is not projectively equivalent to any  
flat affine connection. 
\end{thm} 
\begin{proof}
A set of simple roots of $\sll(n, \bH)$ is given by $\{\alpha_1, \ldots, \alpha_{n-1}\}$. 
Let $\Lambda' = \{\alpha_{i_1}, \cdots, \alpha_{i_j}\}$  be a proper subset of $\Lambda$.  
Then the parabolic subalgebra $\q_{\Lambda'}$ has the basis 
\begin{eqnarray*}
&& \{H^1, \ldots, H^{n-1} \mid i E_{tt}, j E_{tt}, k E_{tt} (1 \leq  t \leq n) \mid \cdots \}. 
\end{eqnarray*} 
The induced affine connection $\nabla$ on $\q_{\Lambda'}$ induces a Lie algebra representation 
$f: \q_{\Lambda'} \to \gl(m+1, \R)$.  
By the straight forward calculation we obtain the following: 
\begin{eqnarray*} 
&&  
\nabla_{iE_{t t}}H^k = 
\left\{
\begin{array}{l}
\frac{i}{n}(n-k)E_{t t}  \ \mbox{for} \ t \leq k  \\  
\frac{i}{n}(-k)E_{t t}  \  \mbox{for} \ t \geq k +1,   
\end{array}\right. \\
&& \nabla_{iE_{t t}}iE_{s s} = \delta_{ts} (H^{t-1} - H^t),   \\
&& \nabla_{iE_{t t}}jE_{s s} = \delta_{ts} k E_{t t},   \\
&& \nabla_{iE_{t t}}kE_{s s} = \delta_{ts} (-j) E_{t t},   \\
&& \gamma(i E_{t t}, i E_{t t}) = \frac{1}{n}.    
\end{eqnarray*} 
These data yields the following: 
\[
f(i E_{t t}) 
= \frac{1}{n} 
\left( {\arraycolsep=0.5mm
\begin{array}{cccccccc|cccccc|c|c} 
&&&&&&&    &&  & 0 &&   &&      &  \\
&&&&&&&    &&  & \vdots &&  && &  \\
&&&&&&&    &&  & 0 &&   &&      &  \\
&&&\mbox{\smash{\huge 0}}&&&&    &\mbox{\smash{\Huge 0}}&  & n H^{t-1} && \mbox{\smash{\huge 0}} &    &
\mbox{\smash{\huge 0}}       & 0  \\ 
&&&&&&&    &&  &  -n H^{t}  &&   &&       &  \\
&&&&&&&    &&  &  0  &&    &&    &   \\
&&&&&&&    &&  &  \vdots &&  &&  & \\
&&&&&&&    &&  & 0  &&    &&    &  \\ 
\hline 
&&&&&&&    &&  &  0  &&   &&   & 0 \\ 
&&& \mbox{\smash{\huge 0}} &&&&    &&  &   \vdots  &&    &&   &  \vdots \\ 
&&&&&&&    &&  &  0  &&    &&   & 0 \\ 
-1&-2&\cdots &-(t-1)&n-t&n-(t+1)&\cdots & 1   &  &  &  0  &&  &   &   \mbox{\smash{\huge 0}}   & n \\
&&&&&&&    &&  &  0 &  &   -n j E_{t t} &&    & 0 \\ 
&&& \mbox{\smash{\huge 0}} &&&&    &&  &  \vdots & n k E_{t t} &&    &   & \vdots \\ 
&&&&&&&    &&  &  0 &   &    &&   &  0  \\
\hline
&&&&&&&    &&  &  0  &&  & &      &  \\ 
&&&\mbox{\smash{\huge 0}} &&&&    && \mbox{\smash{\huge 0}} \ &  \vdots  &&\mbox{\smash{\huge 0}}&    &\mbox{\smash{\huge *}}    & 0  \\
&&&&&&&    &&  &  0 &&  & &  & \\
\hline
&&&0&&&&    && 0 &  -1 &  & 0 &   &0  & 0 
\end{array}}
\right).  
\]

Denote by $\Xi$ a row vector 
\begin{eqnarray*}
(\xi_1, \xi_2, \cdots, \xi_{n-1} \mid \beta_{i i}, \gamma_{i i}, \eta_{i i} \ (1 \leq i \leq n),   
\alpha_{i_r i_{r+1}}, \beta_{i_r i_{r+1}}, \gamma_{i_r i_{r+1}}, \eta_{i_r i_{r+1}}, \\
\alpha_{i_{r+1}, i_r}, \beta_{i_{r+1}, i_r}, \gamma_{i_{r+1}, i_r}, \eta_{i_{r+1}, i_r} (1 \leq  r \leq j) 
\mid \cdots)
\end{eqnarray*} 
of the length $m = \dim \q_{\Lambda'}$. 
Denote by $P$ a $(m+1) \times (m+1)$ matrix    
$
\begin{pmatrix} 
I_m & 0 \\
-\Xi & 1 
\end{pmatrix}$.   
The $\g_1$ component of $n P^{-1} f(i E_{t t}) P$ is equal to 
\begin{eqnarray*} 
(-\beta_{t t}, -2\beta_{t t}, \cdots, - (t-1)\beta_{t t}, (n-t)\beta_{t t}, [n-(t+1)]\beta_{t t}, \cdots, \beta_{t t} 
\mid 0, \cdots, 0, n\xi_{t-1} - n\xi_t -1, \\ 
n\eta_{t t}, -n\gamma_{t t}, 0 \cdots 0 \mid \cdots) 
- n\beta_{t t}(\xi_1, \xi_2, \cdots, \xi_{n-1} \mid \beta_{t t}, \gamma_{t t}, \eta_{t t} (1 \leq s \leq n) \mid \cdots).  
\end{eqnarray*} 
Thus if a representation $P^{-1} f P$ is corresponding to a flat affine connection, then we must have  

\noindent 
(1) $-n \beta_{t t} \xi_t + (n-t)\beta_{t t} = 0$,   (2) $-n \beta_{t t}^2 + n \xi_{t-1} -n \xi_t -1 = 0$ for $1 \leq  t \leq n$.  
Here $\xi_0$ and $\xi_n$ are equal to $0$.   
We now consider whether these equations have a common solution or not. 
Let $\Xi$ be a solution of the equations. 

Firstly suppose that $\beta_{11} \neq 0$. Then from (1) with $t = 1$ we obtain $\xi_1 = \frac{n-1}{n}$. 
Combining this condition with (2) yields $\beta_{11}^2 = -1$. This is a contradiction, hence we must have $\beta_{11} = 0$ and 
$\xi_1 = -\frac{1}{n}$. 

Now we show that if we have $\beta_{1 1} = \beta_{2 2} = \cdots = \beta_{t-1 t-1} = 0$ and 
$\xi_{1} = -\frac{1}{n}, \xi_2 = -\frac{2}{n}, \cdots, \xi_{t-1} = -\frac{t-1}{n}$, then we must have 
$\beta_{t t} = 0$ and $\xi_t = -\frac{t}{n}$.  
Suppose that $\beta_{t t} \neq 0$. Then from the equation 
(1) $-n \beta_{t t} \xi_t + (n-t)\beta_{t t} = 0$ we obtain $\xi_t = \frac{n-t}{n}$. 
Combining this with (2) yields $\beta_{t t}^2 = -1$, which is a contradiction. Hence we must have $\beta_{t t} = 0$. 
Then by using (2) again we obtain $\xi_t = -\frac{t}{n}$.  
It follows that by induction on $1 \leq t \leq n-1$ we obtain $(\xi_1, \xi_2, \cdots, \xi_{n-1}) = 
(-\frac{1}{n}, -\frac{2}{n}, \cdots, -\frac{n-1}{n})$.  
Then from the equation (2) with $t = n$, we obtain $\beta_{n n}^2 = -1$. 
Therefore there are no solutions in the real field, which gives our assertion.  
\end{proof}
\bigskip
\begin{center}
\bf Acknowledgments
\end{center}

The author wishes to thank Prof. Hiroshi Tamaru for giving me the article \cite{tamaru}.

\end{document}